\documentclass[psamsfonts,reqno]{amsart}
\usepackage{amssymb,eucal,graphics,latexsym,xy}
\xyoption{all}
\xyoption{ps}
\usepackage{epstopdf}

\theoremstyle{plain}

\newtheorem{lemma}{Lemma}[section]
\newtheorem{theorem}[lemma]{Theorem}
\newtheorem{proposition}[lemma]{Proposition}
\newtheorem{corollary}[lemma]{Corollary}

\newtheorem*{stat}{\name}
\newcommand{\name}{testing}

\theoremstyle{definition}
\newtheorem{definition}[lemma]{Definition}
\newtheorem{problem}{Problem}

\theoremstyle{remark}

\newtheorem{notation}[lemma]{Notation}

\newenvironment{all}[1]{\renewcommand{\name}{#1}\begin{stat}}
                         {\end{stat}}

\newcommand{\qedc}{{\qed}~{\rm Claim~{\theclaim}.}}
\newcommand{\qedsc}{{\qed}~{\rm Claim.}}

\newcommand{\case}[1]%
{\smallskip\noindent\textbf{\textit{Case}\ {#1}.}}

\numberwithin{equation}{section}
\numberwithin{figure}{section}

\newcommand{\pup}[1]{\textup{(}{#1}\textup{)}}

\newcommand{\bx}{{\boldsymbol{x}}}

\newcommand{\bA}{{\boldsymbol{A}}}
\newcommand{\bB}{{\boldsymbol{B}}}

\newcommand{\bD}{{\boldsymbol{D}}}

\newcommand{\bG}{{\boldsymbol{G}}}
\newcommand{\bL}{{\boldsymbol{L}}}
\newcommand{\bM}{{\boldsymbol{M}}}

\newcommand{\bS}{{\boldsymbol{S}}}

\newcommand{\bX}{{\boldsymbol{X}}}

\newcommand{\xe}{\boldsymbol{e}}
\newcommand{\xf}{\boldsymbol{f}}
\newcommand{\xu}{\boldsymbol{u}}

\newcommand{\les}{\leqslant}
\DeclareMathOperator{\kur}{kur}
\DeclareMathOperator{\brd}{br}
\DeclareMathOperator{\wdt}{wd}
\DeclareMathOperator{\card}{card}
\DeclareMathOperator{\cf}{cf}
\newcommand{\crit}[2]{\mathrm{crit}({#1};{#2})}

\DeclareMathOperator{\Con}{Con}
\DeclareMathOperator{\Conc}{Con_c}

\DeclareMathOperator{\rV}{V}

\DeclareMathOperator{\Id}{Id}
\DeclareMathOperator{\J}{J}

\DeclareMathOperator{\ari}{ar}
\newcommand{\Pow}{\mathfrak{P}}

\newcommand{\res}{\mathbin{\restriction}}

\newcommand{\cA}{\mathcal{A}}
\newcommand{\cB}{\mathcal{B}}
\newcommand{\cC}{\mathcal{C}}
\newcommand{\cD}{\mathcal{D}}

\newcommand{\cS}{\mathcal{S}}

\newcommand{\cV}{\mathcal{V}}

\newcommand{\jirr}{join-ir\-re\-duc\-i\-ble}

\newcommand{\eps}{\varepsilon}
\newcommand{\es}{\varnothing}
\newcommand{\into}{\hookrightarrow}

\newcommand{\mono}{\rightarrowtail}

\newcommand{\famm}[2]{(#1\mid#2)}
\newcommand{\set}[1]{\{#1\}}
\newcommand{\setm}[2]{\set{#1\mid#2}}

\newcommand{\ol}[1]{\overline{#1}}

\newcommand{\two}{\mathbf{2}}

\newcommand{\dnw}{\mathbin{\downarrow}}
\newcommand{\ddnw}{\mathbin{\downdownarrows}}

\newcommand{\upw}{\mathbin{\uparrow}}

\newcommand{\Upw}{\mathbin{\Uparrow}}
\newcommand{\sor}{\mathbin{\triangledown}}
\newcommand{\Sor}{\mathbin{\bigtriangledown}}

\newcommand{\jz}{$(\vee,0)$}
\newcommand{\jzu}{$(\vee,0,1)$}

\newcommand{\jzs}{\jz-semi\-lat\-tice}
\newcommand{\jzus}{\jzu-semi\-lat\-tice}
\newcommand{\jzh}{\jz-ho\-mo\-mor\-phism}

\newcommand{\jzuh}{\jzu-ho\-mo\-mor\-phism}

\newcommand{\jzue}{\jzu-em\-bed\-ding}

\newcommand{\js}{join-sem\-i\-lat\-tice}
\newcommand{\ajs}{al\-most join-sem\-i\-lat\-tice}
\newcommand{\pjs}{pseu\-do join-sem\-i\-lat\-tice}

\newcommand{\rB}{\mathrm{B}}

\begin{document}

\title[Infinite combinatorial issues]{Infinite combinatorial issues raised by lifting problems in universal algebra}

\author[F.~Wehrung]{Friedrich Wehrung}
\address{LMNO, CNRS UMR 6139\\
D\'epartement de Math\'ematiques, BP 5186\\
Universit\'e de Caen, Campus 2\\
14032 Caen cedex\\
France}
\email{wehrung@math.unicaen.fr, fwehrung@yahoo.fr}
\urladdr{http://www.math.unicaen.fr/\~{}wehrung}

\date{\today}



\begin{abstract}
The \emph{critical point} between varieties~$\cA$ and~$\cB$ of algebras is defined as the least cardinality of the semilattice of compact congruences of a member of~$\cA$ but of no member of~$\cB$, if it exists. The study of critical points gives rise to a whole array of problems, often involving lifting problems of either diagrams or objects, with respect to functors. These, in turn, involve problems that belong to infinite combinatorics. We survey some of the combinatorial problems and results thus encountered. The corresponding problematic is articulated around the notion of a \emph{$k$-ladder} (for proving that a critical point is \emph{large}), \emph{large free set theorems} and the classical notation $(\kappa,r,\lambda)\rightarrow m$ (for proving that a critical point is \emph{small}). In the middle, we find \emph{$\lambda$-lifters} of posets and the relation $(\kappa,{<}\lambda)\leadsto P$, for infinite cardinals~$\kappa$ and~$\lambda$ and a poset~$P$.
\end{abstract}

\maketitle

\section{Introduction}\label{S:Intro}
The main goal of the present paper is to give a rough survey of some \emph{combinatorial}, often \emph{set-theoretical} issues raised by recent work on representation problems in \emph{universal algebra}. It is an expanded version of the author's talk at the BLAST~2010 conference in Boulder, Colorado.

Let us first quickly describe the universal algebraic background. A \emph{signature} (or \emph{similarity type}) consists of a set~$\Sigma$, together with a map $\ari\colon\Sigma\to\set{0,1,2,\dots}$ (the \emph{arity map}). An \emph{algebra} is a structure $\bA=(A,\famm{f^\bA}{f\in\Sigma})$, where~$A$ is a nonempty set and for each $f\in\Sigma$, the ``fundamental operation'' $f^\bA$ is a map from~$A^{\ari(f)}$ to~$A$, identified with the unique constant in its range in case~$\ari(f)=0$. A \emph{congruence} of~$\bA$ is an equivalence relation on~$A$ which is compatible with all the fundamental operations of~$\bA$. The set~$\Con\bA$ of all congruences of~$\bA$ forms an \emph{algebraic lattice} under set inclusion, and we denote by $\Conc\bA$ its semilattice of compact elements. Hence~$\Conc\bA$ is the \jzs\ of all finitely generated congruences of~$\bA$. It is a central question in universal algebra to determine which algebraic lattices can be represented as~$\Con\bA$, for an algebra~$\bA$ in a given class of algebras. As it turns out that infinite \emph{\jzs s} are often more conveniently dealt with than infinite algebraic \emph{lattices}, the class of questions above can be rephrased as which \jzs s can be represented as~$\Conc\bA$, for an algebra~$\bA$ in a given class of algebras. A \emph{variety of algebras} is the class of all the algebras satisfying a given set of identities (in a given signature).

Many combinatorial features of the work under discussion in the present paper are contained in the notion of \emph{critical point} between varieties of algebras. The following definition was introduced in the survey paper T\r{u}ma and Wehrung~\cite{CLPSurv}, then slightly amended in Gillibert~\cite{GillTh,Gill1}. For a class~$\cC$ of algebras, we set
 \[
 \Conc\cC:=\setm{\bS}{(\exists\bA\in\cC)(\bS\cong\Conc\bA)}\,,
 \]
the \emph{compact congruence class} of~$\cC$.

\begin{definition}\label{D:CritPt}
The \emph{critical point} $\crit{\cA}{\cB}$, between classes~$\cA$ and~$\cB$ of algebras, is the least possible cardinality of a member of $(\Conc\cA)\setminus(\Conc\cB)$ if $\Conc\cA\not\subseteq\Conc\cB$, and~$\infty$ if~$\Conc\cA\subseteq\Conc\cB$).
\end{definition}

Thus the critical point measures the ``containment defect'' of~$\Conc\cA$ into~$\Conc\cB$. It turns out that estimating the critical point between varieties of algebras, in particular varieties of lattices, is a highly nontrivial task, of which we shall give a rough outline in Section~\ref{S:CritPt}. Proving that a critical point is \emph{large} requires tools called \emph{ladders}; we present a foretaste of the methods used in Sections~\ref{S:ExpleUA} (for $1$-ladders) and~\ref{S:Al1} (for $2$-ladders). Although no universal-algebraic use of $3$-ladders has been found so far (in my optimistic moments I do hope that there will be some), these objects raise fascinating set-theoretical questions. Furthermore, the failure of three-dimensional amalgamation is, in some sense, at the basis of nontrivial counterexamples (cf. R\r{u}\v{z}i\v{c}ka, T\r{u}ma, and Wehrung~\cite{RTW} or Wehrung~\cite{CLP}), and one cannot completely discard the possibility of future uses of that property, see Section~\ref{S:Con3}.

On the other hand, proving that a critical point is \emph{small} involves quite different combinatorial tools, articulated around the notion of a \emph{$\lambda$-lifter} (Section~\ref{S:CLL}) and the more amenable relation $(\kappa,{<}\lambda)\leadsto P$, for infinite cardinals~$\kappa$ and~$\lambda$ and a poset~$P$ (cf. Section~\ref{S:Kurat}). It turns out that the relation $(\kappa,{<}\lambda)\leadsto P$, when specialized to \emph{truncated cubes}~$P$, gives a relation known to infinite combinatorists, namely the $(\kappa,r,\lambda)\rightarrow m$ relation (cf. Notation~\ref{Not:arrRel} and Proposition~\ref{P:rightarr2leadsto}). The more algebraic viewpoint brought by the $(\kappa,{<}\lambda)\leadsto P$ framework made it possible, in Gillibert and Wehrung~\cite{GiWe1}, to establish new relations of the form $(\kappa,r,\lambda)\rightarrow m$, of which we give an overview in Section~\ref{S:LargeFree}.

In the middle of ``large critical point results'' (involving ladders) or ``small critical point results'' (involving the relation $(\kappa,{<}\lambda)\leadsto P$), there is a work of category theory (cf. Gillibert and Wehrung~\cite{Larder}), introducing the categorical notion of a \emph{larder} and the combinatorial notion of a \emph{lifter} (cf. Section~\ref{S:CLL}). These tools are mainly designed to introduce, and use, the \emph{Condensate Lifting Lemma} CLL, and its main precursor the \emph{Armature Lemma}. These results make it possible to turn \emph{diagram counterexamples} to \emph{object counterexamples}. While there are surprisingly many statements that can be proved that way (beginning with those involving critical points), we present in Section~\ref{S:BowTie} a quite isolated diagram counterexample, due to T\r{u}ma and the author, for which there is still no known associated object counterexample.
 
\section{Basic concepts}\label{S:Basic}
\subsection{Set theory}\label{Su:BasicSet}
We shall use basic set-theoretical notation and terminology about ordinals and cardinals. 
We denote by~$f``X$ the image of a set~$X$ under a function~$f$. Cardinals are initial ordinals. We denote by~$\cf\alpha$ the cofinality of an ordinal~$\alpha$. We denote by $\omega:=\set{0,1,2,\dots}$ the first limit ordinal, mostly denoted by $\aleph_0$ in case it is viewed as a cardinal. We identify every non-negative integer~$n$ with the finite set $\set{0,1,\dots,n-1}$ (so $0=\es$). We denote by~$\kappa^+$ the successor cardinal of a cardinal~$\kappa$, and we define~$\kappa^{+n}$, for a non-negative integer~$n$, by $\kappa^{+0}:= \kappa$ and $\kappa^{+(n+1)}=(\kappa^{+n})^+$.

We denote by~$\Pow(X)$ the powerset of a set~$X$, and we set
 \begin{align*}
 [X]^\kappa&:=\setm{Y\in\Pow(X)}{\card Y=\kappa}\,,\\
 [X]^{<\kappa}&:=\setm{Y\in\Pow(X)}{\card Y<\kappa}\,,\\
 [X]^{{\les}\kappa}&:=\setm{Y\in\Pow(X)}{\card Y\leq\kappa}\,,
 \end{align*}
for every cardinal~$\kappa$.

By ``countable'' we shall always mean ``at most countable''.

\subsection{Partially ordered sets (posets)}\label{Su:Posets}
All our posets will be nonempty. For posets~$P$ and~$Q$, a map $f\colon P\to Q$ is \emph{isotone} if $x\leq y$ implies that $f(x)\leq f(y)$, for all $x,y\in P$.

We denote by~$0_P$ the least element of~$P$ if it exists. An element~$p$ in a poset~$P$ is \emph{\jirr} if $p=\bigvee X$ implies that $p\in X$, for every (possibly empty) finite subset~$X$ of~$P$; we denote by $\J(P)$ the set of all \jirr\ elements of~$P$, endowed with the induced partial ordering. We set
 \begin{align*}
 Q\dnw X&:=\setm{q\in Q}{(\exists x\in X)(q\leq x)}\,,\\
 Q\ddnw X&:=\setm{q\in Q}{(\exists x\in X)(q<x)}\,,\\
 Q\Upw X&:=\setm{q\in Q}{(\forall x\in X)(x\leq q)}\,
  \end{align*}
for all subsets~$Q$ and~$X$ of~$P$; in case $X=\set{a}$ is a singleton, then we shall write $Q\dnw a$ instead of~$Q\dnw\set{a}$, and so on. We say that~$P$ is \emph{lower finite} if $P\dnw a$ is finite for each $a\in P$. A subset~$Q$ of~$P$ is a \emph{lower subset of~$P$} if~$P\dnw Q=Q$. We say that~$P$ is a \emph{tree} if~$P$ has a smallest element and~$P\dnw a$ is a chain for each $a\in P$. An \emph{ideal} of a poset~$P$ is a nonempty, upward directed, lower subset of~$P$.

We recall the definitions of a few classical numerical invariants associated to a finite poset~$P$:
\begin{itemize}
\item[---] The \emph{width} of~$P$, denoted by $\wdt(P)$, is the maximal cardinality of a pairwise incomparable subset of~$P$.

\item[---] The \emph{breadth} of~$P$, denoted by $\brd(P)$, is the least natural number~$m$ such that for all $x_i,y_i\in P$, for $0\leq i\leq m$, if $x_i\leq y_j$ for all $i\neq j$, then there exists~$i$ such that $x_i\leq y_i$.

\item[---] The \emph{order-dimension} of~$P$, denoted by $\dim(P)$, is the least possible natural number~$m$ such that~$P$ embeds, as a poset, into a product of~$m$ chains.
\end{itemize}

A \jzs~$\bS$ is \emph{distributive} if for all $a,b,c\in S$, if $c\leq a\vee b$, then there are $x\leq a$ and $y\leq b$ in~$S$ such that $c=x\vee y$. Equivalently, the ideal lattice~$\Id\bS$ of~$\bS$ is a distributive lattice (cf. Gr\"atzer~\cite[Section~II.5]{GLT2}).

\section{An example in universal algebra: representing countable distributive semilattices}\label{S:ExpleUA}

Let us start with an example. We are given a class~$\cC$ of algebras, and we assume that every \emph{finite Boolean} lattice is isomorphic to~$\Conc\bA$ for some $\bA\in\cC$. Can we extend this representation result to \emph{countable} distributive \jzs s?

Without additional assumptions, this cannot be done. For example, let $\cC:=\cD$ be the variety of all \emph{distributive lattices}. Then every finite Boolean lattice~$\bB$ belongs to~$\Conc\cD$ (for $\Conc\bB\cong\bB$), but for any $\bD\in\cD$, the \jzs\ $\Conc\bD$ is generalized Boolean, thus it is a lattice; and many distributive \jzs s (even countable ones) are not lattices.

Can we find reasonable assumptions, on the class~$\cC$, that make it possible to represent any \emph{countable} distributive \jzs, starting with a \emph{finite} representability result?

Here is the required finite representability result.

\begin{all}{One-dimensional amalgamation property for $\Conc$}
\textup{
For each $\bA\in\cC$, each finite Boolean semilattice~$\bS$, and each \jzh\ $\varphi\colon\Conc\bA\to\bS$, there are~$\bB\in\cC$, a homomorphism $f\colon\bA\to\bB$, and an isomorphism $\eps\colon\Conc\bB\to\bS$ such that $\varphi=\eps\circ\Conc f$.}
\end{all}

The statement above involves a yet undefined notation, namely~$\Conc f$, where~$f$ is no longer an \emph{object}, but a \emph{morphism}. Nor surprisingly, $\Conc f$ is defined as the map from~$\Conc\bA$ to~$\Conc\bB$ that to every compact congruence~$\alpha$ of~$\bA$ associates the congruence of~$\bB$ generated by $\setm{(f(x),f(y))}{(x,y)\in\alpha}$. The map $\Conc f$ is a \jzh\ from~$\Conc\bA$ to~$\Conc\bB$. We obtain the missing piece of information that is necessary in order to understand the ``amalgamation'' statement above, with a plus about \emph{direct limits} (called, in category theory, \emph{directed colimits}). The one-dimensional amalgamation property is illustrated in Figure~\ref{Fig:Con1}.

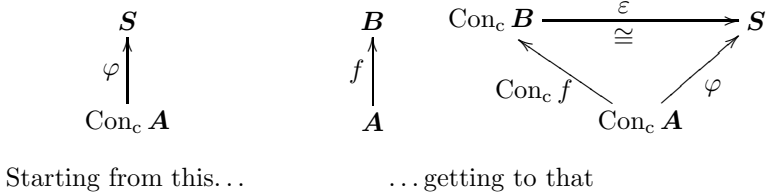
\begin{figure}[htb]
 \[
\xymatrixrowsep{2pc}\xymatrixcolsep{1.5pc}
\def\labelstyle{\displaystyle}
\xymatrix{
\bS &&& \bB & \Conc\bB\ar[rr]^{\eps}_{\cong} && \bS\\
\Conc\bA\ar[u]^{\varphi}
\save+<0ex,-5ex>\drop{\text{Starting from this\dots}}\restore
&&& \bA\ar[u]^f &\save+<0ex,-5ex>\drop{\text{\dots getting to that}}\restore&
\Conc\bA\ar[ul]^{\Conc f}\ar[ur]_{\varphi}&
}
 \]
\caption{One-dimensional amalgamation for $\Conc$}
\label{Fig:Con1}
\end{figure}

\begin{lemma}\label{L:ConcFunct}
The assignment $\bA\mapsto\Conc\bA$, $f\mapsto\Conc f$ defines a \emph{functor} from the class of all algebras with a given signature, with its homomorphisms, to the category of all \jzs s with \jzh s. This functor preserves all direct limits.
\end{lemma}

Now suppose that the one-dimensional amalgamation property holds for the functor~$\Conc$ on the class~$\cC$. Let~$\bS$ be a countable~\jzs. Using a result from Bulman-Fleming and McDowell~\cite{BuMD78} (see also Goodearl and Wehrung~\cite{GoWe01} for a self-contained proof), there is a representation of the form
 \begin{equation}\label{Eq:DirSystS}
 \bS=\varinjlim_{n<\omega}\bS_n\,,
 \end{equation}
where all the~$\bS_n$ are \emph{finite Boolean} \jzs s and all the transition maps of the direct system in~\eqref{Eq:DirSystS} are \jzh s. Suppose that $\bS_n\cong\Conc\bA_n$, with an isomorphism $\eps_n\colon\Conc\bA_n\to\bS_n$, and denote by $\sigma_n\colon\bS_n\to\bS_{n+1}$ the transition map associated with the direct system in~\eqref{Eq:DirSystS}. We get a \jzh\ $\sigma_n\circ\eps_n\colon\Conc\bA_n\to\bS_{n+1}$, thus, by the one-dimensional amalgamation assumption, there are $\bA_{n+1}\in\cC$, a homomorphism $f_n\colon\bA_n\to\bA_{n+1}$, and an isomorphism $\eps_{n+1}\colon\Conc\bA_{n+1}\to\bS_{n+1}$ such that $\eps_{n+1}\circ\Conc f_n=\sigma_n\circ\eps_n$. Defining~$\bA$ as the direct limit of the system
 \[
 \xymatrix{
 \bA_0\ar[r]^{f_0} & \bA_1\ar[r]^{f_1} & \bA_2\ar[r]^{f_2} &\cdots\cdots\,,
 }
 \]
we obtain, as the functor~$\Conc$ preserves direct limits, that $\Conc\bA\cong\bS$. So we have obtained the following result.

\begin{proposition}\label{P:Con1toCtble}
Let~$\cC$ be a class of algebras of the same signature. If the one-dimensional amalgamation property holds for~$\Conc$ on~$\cC$, then every countable distributive \jzs\ is isomorphic to~$\Conc\bA$ for the direct limit~$\bA$ of some countable sequence in~$\cC$.
\end{proposition}

This approach is used in R\r{u}\v{z}i\v{c}ka, T\r{u}ma, and Wehrung~\cite{RTW} to establish the following statement.

\begin{corollary}\label{C:Con1lgrps}
Every distributive algebraic lattice with countably many compact elements is isomorphic to the congruence lattice of some lattice-ordered group.
\end{corollary}

The corresponding statement with ``countable'' replaced by ``with at most~$\aleph_1$ elements'' is still open. Yet there are also known positive results in cardinality~$\aleph_1$, for which we shall set the framework in the next section. It will turn out that none of the corresponding representation results holds at cardinality~$\aleph_2$.

\section{One step further: representing semilattices of cardinality~$\aleph_1$}
\label{S:Al1}

In the context of Section~\ref{S:ExpleUA}, one may try to represent larger distributive \jzs s than just the countable ones. Thus the first step consists of trying semilattices of cardinality~$\aleph_1$. The na\"\i ve approach, consisting of keeping the one-dimensional amalgamation property and replacing the chain~$\omega$ of all natural numbers by the larger chain~$\omega_1$ of all countable ordinals, obviously fails: if $\card S=\aleph_1$, then there is no way to represent $\bS=\varinjlim_{\xi<\omega_1}\bS_\xi$ with all~$S_\xi$ finite.

The relevant strengthening of the assumption on~$\cC$ is the following.

\begin{all}{Two-dimensional amalgamation property for $\Conc$}
\textup{Let $\bA_0$, $\bA_1$, $\bA_2$ be members of~$\cC$ with homomorphisms $f_i\colon\bA_0\to\bA_i$ for $i\in\set{1,2}$, and let~$\bS$ be a finite Boolean \jzs\ with \jzh s $\psi_i\colon\Conc\bA_i\to\bS$, for $i\in\set{1,2}$, such that $\psi_1\circ\Conc f_1=\psi_2\circ\Conc f_2$. Then there are a member~$\bA$ of~$\cC$, homomorphisms $g_i\colon\bA_i\to\bA$ for $i\in\set{1,2}$, and an isomorphism $\eps\colon\Conc\bA\to\bS$ such that $g_1\circ f_1=g_2\circ f_2$ while $\psi_i=\eps\circ\Conc g_i$ for each $i\in\set{1,2}$.}
\end{all}

The two-dimensional amalgamation property can be illustrated by Figure~\ref{Fig:Con2}. On that picture, the plain vanilla arrows are those that are assumed to exist from the start, while the dotted ones are those whose existence follows from the property.

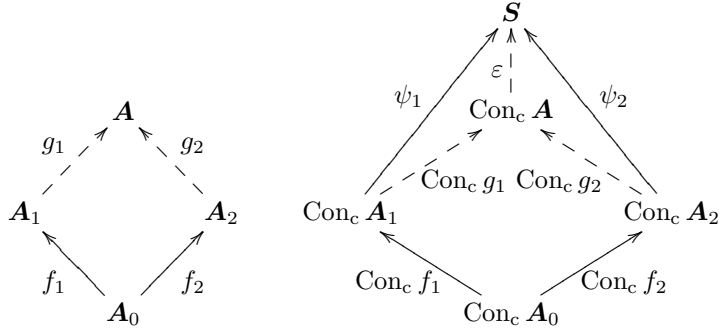
\begin{figure}[htb]
 \[
\xymatrixrowsep{2pc}\xymatrixcolsep{1.5pc}
\def\labelstyle{\displaystyle}
\xymatrix{
& & & & \bS &\\
& \bA & & & \Conc\bA\ar@{-->}[u]^(.4){\eps} \\
\bA_1\ar@{-->}[ru]^{g_1} & & \bA_2\ar@{-->}[lu]_{g_2}
& \Conc\bA_1\ar@{-->}[ur]_{\Conc g_1}\ar[uur]^{\psi_1} & &
\Conc\bA_2\ar@{-->}[ul]^{\Conc g_2}\ar[uul]_{\psi_2}\\
& \bA_0\ar[ul]^{f_1}\ar[ur]_{f_2} & & &
\Conc\bA_0\ar[ul]^{\Conc f_1}\ar[ur]_{\Conc f_2}
}
 \]
\caption{Two-dimensional amalgamation for $\Conc$}
\label{Fig:Con2}
\end{figure}

Assuming two-dimensional congruence amalgamation, we still need to replace the chain~$\omega$ by a poset of cardinality~$\aleph_1$ that makes it possible to use two-dimensional congruence amalgamation. We remind the reader that lower finite posets are defined in Section~\ref{Su:Posets}.

\begin{definition}\label{D:kladd}
Let~$k$ be a positive integer. A lower finite lattice~$\bL$ is a \emph{$k$-ladder} if every element of~$L$ has at most~$k$ lower covers in~$\bL$.
\end{definition}

Note that every $k$-ladder has breadth at most $k$. The $1$-ladders are exactly the finite chains and the chain $\omega$ of all natural numbers.
Note that $k$-ladders are called \emph{$k$-frames} in H.~Dobbertin 
\cite{Dobb86}. Due to a terminology conflict with another lattice-theoretical concept, the \emph{von~Neumann frames}, we chose in Gr\"atzer, Lakser, and Wehrung~\cite{GLW} to call these objects \emph{ladders}. The concept originates in S.\,Z. Ditor~\cite{Dito84}, where $k$-ladders are simply called \emph{$k$-lattices}.

Ditor proves in~\cite{Dito84} that every $k$-ladder has at most~$\aleph_{k-1}$ elements. (This is also a straightforward consequence of Kuratowski's Free Set Theorem, see Kuratowski~\cite{Kura51}.) More relevant to our present congruence representation problem, Ditor proves the following.

\begin{proposition}\label{P:2laddAl1}
There exists a $2$-ladder of cardinality~$\aleph_1$.
\end{proposition}

Now suppose that two-dimensional congruence amalgamation holds. We prove that every distributive \jzs~$\bS$ with at most~$\aleph_1$ elements is isomorphic to~$\Conc\bA$, for some direct limit~$\bA$ of members of~$\cC$. Write $\bS=\varinjlim\vec\bS$, for a directed poset~$P$ with at most~$\aleph_1$ elements and a direct system $\vec\bS=\famm{\bS_p,\sigma_p^q}{p\leq q\text{ in }P}$ of finite Boolean \jzs s and \jzh s. For example, $P$ can be assumed to be the set of all finite subsets of~$S\times\omega$, partially ordered under set inclusion. Let~$F$ be a $2$-ladder of cardinality~$\aleph_1$. As~$F$ is lower finite, it is easy to construct, by induction, an isotone cofinal map $\nu\colon F\to P$; this reduces the situation to the case where $P=F$. Then, using two-dimensional amalgamation and the definition of a $2$-ladder, one constructs, by induction within the ladder, an $F$-indexed direct system~$\vec\bA$ in~$\cC$ such that $\Conc\vec\bA\cong\vec\bS$. Now take~$\bA:=\varinjlim\vec\bA$. As the functor~$\Conc$ preserves direct limits, $\Conc\bA\cong\bS$. We have thus outlined the proof of the following.

\begin{proposition}\label{P:Con2toAl1}
Let~$\cC$ be a class of algebras of the same signature. If the two-dimensional amalgamation property holds for~$\Conc$ on~$\cC$, then every distributive \jzs\ with at most~$\aleph_1$ elements is isomorphic to~$\Conc\bA$ for the direct limit~$\bA$ of a diagram, indexed by a $2$-ladder of cardinality~$\aleph_1$, of members of~$\cC$.
\end{proposition}

This argument has turned out to be very useful in proving various representation results, in size~$\aleph_1$, with respect to the~$\Conc$ functor, starting with A.\,P. Huhn~\cite{Huhn89a, Huhn89b}; see, for example, Gr\"atzer, Lakser, and Wehrung~\cite{GLW}, or R\r{u}\v{z}i\v{c}ka, T\r{u}ma, and Wehrung~\cite{RTW}. See also Wehrung~\cite{WReg} for a representation result with respect to another functor than~$\Conc$, namely the compact ideal semilattice functor on a regular ring.

\section{Looking ahead: $3$-dimensional amalgamation and $3$-ladders}
\label{S:Con3}

For higher cardinalities than~$\aleph_1$, the situation gets somehow romantic. There are, to this date, no known congruence representation results specifically stated in cardinality~$\aleph_2$. Yet there are theorems, now of set-theoretical nature, which are an order of magnitude harder than those about representation in cardinality below~$\aleph_1$. These results go together with an array of still unsolved problems.

We first deal with the easy problem of $3$-dimensional congruence amalgamation, for a class~$\cC$ of algebras with the same signature. Set $B_3:=\Pow(3)$ (recall that $3=\set{0,1,2}$), partially ordered under set inclusion, and set $B_3^=:=B_3\setminus\set{3}$. The statement of three-dimensional amalgamation is quite discouraging at first sight:

\begin{all}{Three-dimensional amalgamation property for $\Conc$}
\textup{Let~$\bS$ be a finite Boolean semilattice, let $\vec\bA=\famm{\bA_p,\alpha_p^q}{p\leq q\text{ in }B_3^=}$ be a $B_3^=$-indexed direct system in~$\cC$, and let $\famm{\sigma_p}{p\in B_3^=}$ be a natural transformation from $\Conc\vec\bA$ to the one-object diagram~$\bS$. Then there are $\bA\in\cC$ and a natural transformation $\famm{\alpha_p}{p\in B_3^=}$ from~$\vec\bA$ to the one-object diagram~$\bA$, together with an isomorphism $\eps\colon\Conc\bA\to\bS$, such that $\sigma_p=\eps\circ\Conc\alpha_p$ for each $p\in B_3^=$.}
\end{all}

This can be explained as follows. We are given a truncated cube (here, a $B_3^=$-indexed diagram)~$\vec\bA$ of members of~$\cC$, and we form the image $\Conc\vec\bA$ of this diagram under the~$\Conc$ functor. At the top of this diagram, we put the finite Boolean semilattice~$\bS$, thus getting a cube~$\vec\bS$ of \jzs s. Three-dimensional amalgamation states that there is a way to adjoin a top vertex to~$\vec\bA$, thus getting a cube~$\vec\bB$, such that $\Conc\vec\bB\cong\vec\bS$, in a way that preserves the already existing relations. This is illustrated in Figure~\ref{Fig:Con3}, with $\ol{\alpha}_p:=\Conc\alpha_p$ for each $p\in\set{\set{0,1},\set{0,2},\set{1,2}}$. The unlabeled arrows should be obvious: for example, the arrow from~$\bA_{\set{0}}$ to~$\bA_{\set{0,1}}$ is~$\alpha_{\set{0}}^{\set{0,1}}$; and so on.

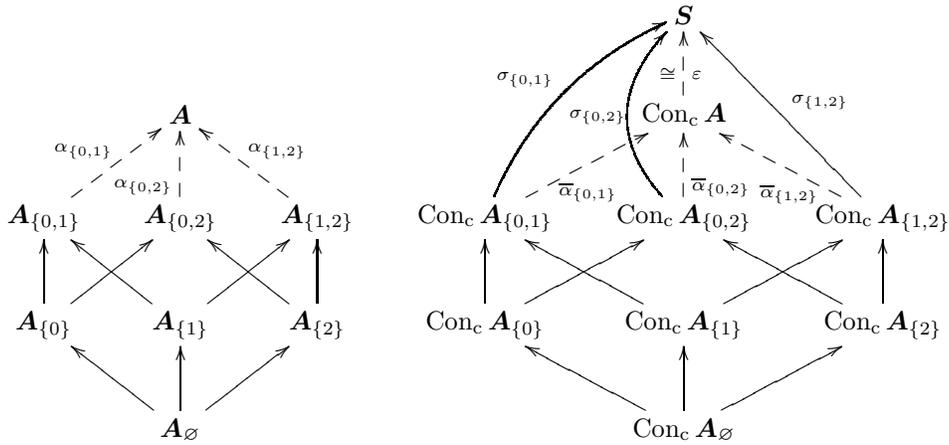
\begin{figure}[htb]
 \[
\xymatrixrowsep{2pc}\xymatrixcolsep{1.5pc}
\xymatrix{
&&&&\bS & \\
& \bA &&& \Conc\bA\ar@{-->}[u]_(.4){\eps}^(.4){\cong} &\\
\bA_{\set{0,1}}\ar@{-->}[ur]^{\alpha_{\set{0,1}}} 
& \bA_{\set{0,2}}\ar@{-->}[u]^(.3){\alpha_{\set{0,2}}} 
& \bA_{\set{1,2}}\ar@{-->}[ul]_{\alpha_{\set{1,2}}} &
\Conc\bA_{\set{0,1}}\ar@{-->}[ur]_(.42){\ol{\alpha}_{\set{0,1}}}
\ar@/^1pc/[uur]^{\sigma_{\set{0,1}}}
& \Conc\bA_{\set{0,2}}\ar@{-->}[u]_(.3){\ol{\alpha}_{\set{0,2}}}
\ar@/^1.8pc/[uu]^(.5){\sigma_{\set{0,2}}\!\!}
& \Conc\bA_{\set{1,2}}\ar@{-->}[ul]^(.4){\ol{\alpha}_{\set{1,2}}\!\!}
\ar[uul]_{\sigma_{\set{1,2}}}\\
\bA_{\set{0}}\ar[u]\ar[ur] & \bA_{\set{1}}\ar[ul]\ar[ur] & \bA_{\set{2}}\ar[ul]\ar[u] &
\Conc\bA_{\set{0}}\ar[u]\ar[ur] & \Conc\bA_{\set{1}}\ar[ul]\ar[ur]
& \Conc\bA_{\set{2}}\ar[ul]\ar[u]\\
& \bA_{\es}\ar[ul]\ar[u]\ar[ur] &&& \Conc\bA_{\es}\ar[ul]\ar[u]\ar[ur] &
}
 \]
\caption{Three-dimensional amalgamation for $\Conc$}
\label{Fig:Con3}
\end{figure}

Unfortunately, the following result shows that three-dimensional amalgamation occurs only in exceptional cases.

\begin{proposition}\label{P:3Amalg}
Suppose that a class~$\cC$ of algebras of the same signature satisfies three-dimensional amalgamation. Let~$\bA_0$, $\bA_1$, $\bA_2$ be members of~$\cC$ with one-to-one homomorphisms $f,g\colon\bA_1\into\bA_2$. We suppose that~$\bA_0$ is a substructure of~$\bA_1$. If $f\res_{\bA_0}=g\res_{\bA_0}$, then $f=g$.
\end{proposition}

\begin{proof}
Assume that $f\res_{\bA_0}=g\res_{\bA_0}$.
Set $\two:=\set{0,1}$, and for each $i\in\set{1,2}$, denote by $\sigma_i\colon\Conc\bA_i\to\two$ the unique \jzh\ that sends every nonzero element to~$1$. We apply three-dimensional amalgamation to the diagrams represented in Figure~\ref{Fig:Con31}, minus~$\bA$, $\Conc\bA$, and the dotted arrows. Note that the one-to-oneness of both~$f$ and~$g$ ensures that both maps $\ol{f}:=\Conc f$ and $\ol{g}:=\Conc g$ separate zero, thus the right hand side diagram commutes. Three-dimensional amalgamation ensures the existence of~$\bA$ together with homomorphisms $u,v\colon\bA_1\to\bA$ and $w\colon\bA_2\to\bA$ with an isomorphism $\eps\colon\Conc\bA\to\two$ such that both diagrams of Figure~\ref{Fig:Con31} (dotted arrows included) commute. Again, we are using the notational convention $\ol{w}:=\Conc w$, and so on. The unlabeled arrows should be obvious, for example the arrow from~$\bA_0$ to~$\bA_1$ is the inclusion map.

\begin{figure}[htb]
 \[
\xymatrixrowsep{2pc}\xymatrixcolsep{1.5pc}
\def\labelstyle{\displaystyle}
\xymatrix{
&&&&\two & \\
& \bA &&& \Conc\bA\ar@{-->}[u]_(.4){\eps}^(.4){\cong} &\\
\bA_1\ar@{-->}[ur]^u
& \bA_2\ar@{-->}[u]^(.4)w
& \bA_1\ar@{-->}[ul]_v &
\Conc\bA_1\ar@{-->}[ur]_(.42){\ol{u}}
\ar@/^1pc/[uur]^{\sigma_1}
& \Conc\bA_2\ar@{-->}[u]_(.3){\ol{w}}
\ar@/^1.8pc/[uu]^(.5){\sigma_2\!}
& \Conc\bA_1\ar@{-->}[ul]^(.42){\ol{v}}
\ar[uul]_{\sigma_1}\\
\bA_1\ar@{=}[u]\ar@{^(->}[ur]_(.35){\!\!\!f} & \bA_1\ar@{=}[ul]\ar@{=}[ur] 
& \bA_1\ar@{_(->}[ul]^(.35){g\!\!}\ar@{=}[u] &
\Conc\bA_1\ar@{=}[u]\ar[ur]_(.4){\!\!\!\ol{f}} & \Conc\bA_1\ar@{=}[ul]\ar@{=}[ur]
& \Conc\bA_1\ar[ul]^(.4){\ol{g}\!\!}\ar@{=}[u]\\
& \bA_0\ar@{_(->}[ul]\ar@{_(->}[u]\ar@{^(->}[ur] &&& \Conc\bA_0\ar[ul]\ar[u]\ar[ur] &
}
 \]
\caption{A special case of three-dimensional amalgamation}
\label{Fig:Con31}
\end{figure}
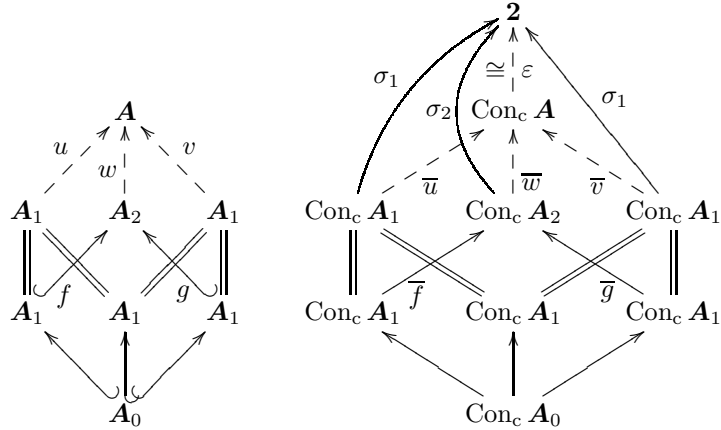

{}From the commutativity of the right hand side diagram it follows that $\sigma_2=\eps\circ\ol{w}$. As~$\sigma_2$ separates zero and~$\eps$ is an isomorphism, $\ol{w}$ separates zero, which means that~$w$ is one-to-one. {}From the commutativity of the left hand side diagram it follows that $u=w\circ f$, $v=w\circ g$, and $u=v$; so $w\circ f=w\circ g$. As~$w$ is one-to-one, we obtain that $f=g$.
\end{proof}

Proposition~\ref{P:3Amalg} suggests at first sight that it is quite hopeless to achieve any representation proof in cardinality~$\aleph_2$ \emph{via} three-dimensional amalgamation. Although some hope remains to apply three-dimensional amalgamation on special configurations, such as \emph{strong amalgamation}, no such application has been found so far.

Even in case, some day, someone manages to work out a usable version of three-dimensional amalgamation, another catch is the \emph{existence of ladders}. Indeed, in order to be able to extend the proof of Proposition~\ref{P:Con2toAl1} to classes satisfying some version of three-dimensional amalgamation, we need the existence of a $3$-ladder of cardinality~$\aleph_2$. This problem was first raised in Ditor~\cite{Dito84}.

\begin{all}{Ditor's Problem}
Does there exist a $3$-ladder of cardinality~$\aleph_2$?
\end{all}

The following result, by Wehrung~\cite{3ladd}, shows that a positive solution of Ditor's Problem is at least consistent with~$\mathsf{ZFC}$.

\begin{theorem}\label{T:Ex3Ladd}
Suppose that either the axiom $\mathsf{MA}(\aleph_1;\text{precaliber }\aleph_1)$ holds or there exists a gap-$1$ morass. Then there exists a $3$-ladder of cardinality~$\aleph_2$.
\end{theorem}

In particular, as every $\mathbf{L}[A]$, for $A\subseteq\omega_1$, has a gap-$1$ morass (cf. Devlin \cite{Devl84}), it follows that \emph{If there is no $3$-ladder of cardinality~$\aleph_2$, then~$\omega_2$ is inaccessible in the constructible universe~$\mathbf{L}$}.

The axiom $\mathsf{MA}(\aleph_1;\text{precaliber }\aleph_1)$, which is Martin's Axiom restricted to collections of~$\aleph_1$ dense subsets in posets of precaliber~$\aleph_1$ (cf. Weiss \cite[Section~3]{Weiss}), holds in a generic extension of the universe by a suitable notion of forcing satisfying the \emph{countable chain condition} (cf. Jech \cite[Section~16]{Jech03}). On the other hand, a gap-$1$ morass exists in a generic extension by a suitable \emph{countably closed} notion of forcing (cf. Velleman~\cite[Section~3.3]{Vell82}; see also Brooke-Taylor and Friedman~\cite[Proposition~25]{BTFr09}). As countably closed posets with the countable chain condition are, from the forcing viewpoint, trivial, the two notions of forcing involved in those generic extensions cannot be the same: they must be, in some sense, ``orthogonal''! This suggests a positive solution to Ditor's Problem in~$\mathsf{ZFC}$, nevertheless no such solution has been found so far.

Our proof of Theorem~\ref{T:Ex3Ladd} uses the equivalent form of morasses called \emph{simplified morasses} introduced in Velleman~\cite{Vell84a}.

It seems reasonable to expect that higher morasses, or higher versions of Martin's Axiom, would imply the existence of $k$-ladders of cardinality~$\aleph_{k-1}$ for $k\geq 4$. In an ideal world, Ditor's Problem would first be settled in~$\mathsf{ZFC}$ alone. However, we are not living in an ideal world\dots

\section{CLL, larders, and lifters}\label{S:CLL}

Propositions~\ref{P:Con1toCtble} and~\ref{P:Con2toAl1} are typical illustrations of a ``finite'' representation result on the functor~$\Conc$ making it possible to prove an ``infinite'' representation result. Now there are situations where the problematic is just opposite: namely, going from an infinite representation result to a finite one. Finite would appear as a limit case of infinite! A paradigm for such situations (\emph{critical points}) will shortly be given in Section~\ref{S:CritPt}.

The arguments underlying Propositions~\ref{P:Con1toCtble} and~\ref{P:Con2toAl1} are best set in a \emph{category theory} context. Basically, we are given \emph{categories}~$\cA$, $\cB$, $\cS$ together with \emph{functors} $\Phi\colon\cA\to\cS$ and~$\Psi\colon\cB\to\cS$. We are trying to find an assignment $\Gamma\colon\cA\to\cB$ such that $\Phi(A)\cong\Psi\Gamma(A)$, \emph{naturally in~$A$}, on a ``large'' subcategory of~$\cA$. We can paraphrase this further by saying that we are trying to make~$\Gamma$ ``as functorial as possible''. Hence we need an assumption of the form ``for many $A\in\cA$, there exists $B\in\cB$ such that $\Phi(A)\cong\Psi(B)$''. The situation is illustrated in Figure~\ref{Fig:GetFctGam}.

\begin{figure}[htb]
 \[
 \def\labelstyle{\displaystyle}
 \xymatrix{
 & \cS & & & & \cS &\\
 \cA\ar[ur]^{\Phi} & 
 \save+<0ex,-5ex>\drop{\text{Starting from this\dots}}\restore
 & \cB\ar[ul]_{\Psi} & & \cA\ar[ur]^{\Phi}\ar[rr]^{\Gamma}
 &
 \save+<0ex,-5ex>\drop{\text{\dots getting to that}}\restore
 & \cB\ar[ul]_{\Psi}
 }
 \]
\caption{Factoring~$\Phi$ through~$\Psi$}
\label{Fig:GetFctGam}
\end{figure}
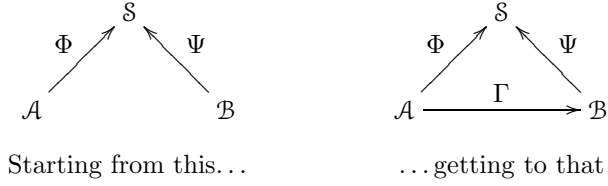

This problem is tackled, in Gillibert and Wehrung~\cite{Larder}, by the introduction of a whole theory, involving a categorical concept called a \emph{larder}. A larder consists of categories~$\cA$, $\cB$, $\cS$ together with functors $\Phi\colon\cA\to\cS$ and~$\Psi\colon\cB\to\cS$, plus a few add-ons, namely full subcategories~$\cA^\dagger\subseteq\cA$ and $\cB^\dagger\subseteq\cB$ together with a subcategory $\cS^{\Rightarrow}\subseteq\cS$. These add-ons are required to satisfy a few obvious-looking preservation properties (e.g., directed colimits), plus a less obvious-looking ``L\"owenheim-Skolem type'' property with respect to the functor~$\Psi$. An important additional attribute of a larder, reminiscent of the First Isomorphism Theorem in group theory, is called \emph{projectability}. Most applications of our theory involve a result called the ``Condensate Lifting Lemma'', or CLL, which we paraphrase below.

\begin{all}{Condensate Lifting Lemma}
Suppose that $\cA$, $\cB$, $\cS$, $\Phi$, and~$\Psi$ are part of a projectable larder and let~$P$ be a ``nice'' poset. If
 \[
 (\text{For many }A\in\cA)(\exists B\in\cB)\bigl(\Phi(A)\cong\Psi(B)\bigr)\,,
 \]
then
 \[
 (\text{For many }\vec A\in\cA^P)(\exists\vec B\in\cB^P)
 \bigl(\Phi(\vec A)\cong\Psi(\vec B)\bigr)\,.
 \]
\end{all}

In this statement, $\cA^P$ denotes the category whose objects are all the direct systems $\famm{A_p,\alpha_p^q}{p\leq q\text{ in }P}$, with~$A_p$ an object of~$\cA$ and $\alpha_p^q\colon A_p\to A_q$, with the natural transformations as morphisms. We shall go back to the meaning of ``nice poset'' in a moment.

To paraphrase CLL further, if~$\Psi$ lifts many \emph{objects}, then it lifts many ($P$-indexed) \emph{diagrams}. Yet another way to see this is that CLL turns \emph{diagram counterexamples} to \emph{object counterexamples}. Applications of CLL, or its main precursor the \emph{Armature Lemma}, include the following.

\begin{itemize}
\item[---] Let~$\cV$ be a nondistributive variety of lattices. Then the free lattice (resp., free bounded lattice) in~$\cV$ on~$\aleph_1$ generators has no congruence-permutable, congruence-preserving extension (cf. Gillibert and Wehrung~\cite{Larder}).

\item[---] There exists a non-coordinatizable sectionally complemented modular lattice (without unit), of cardinality~$\aleph_1$, with a large $4$-frame (cf. Wehrung~\cite{Banasch2}).

\item[---] There exists a lattice of cardinality~$\aleph_1$, in the variety generated by the five-element modular non-distributive lattice~$M_3$, without any congruence $n$-permutable, congruence-preserving extension for any positive integer~$n$ (cf. Gillibert \cite{GilCat}).

\item[---] Let~$\cA$ and~$\cB$ be varieties of lattices such that~$\cA$ is contained neither in~$\cB$ nor its dual, and every simple member of~$\cB$ has a prime interval. Then the critical point $\crit{\cA}{\cB}$ (cf. Definition~\ref{D:CritPt}) lies below~$\aleph_2$ (cf. Gillibert~\cite{Gill3}).

\item[---] There is a unital exchange ring~$R$, of cardinality~$\aleph_3$, such that the monoid $\rV(R)$ of all isomorphism types of finitely generated right $R$-modules is not isomorphic to~$\rV(B)$ for any ring~$B$ that is either von~Neumann regular or a C*-algebra of real rank zero (cf. Wehrung~\cite{VLift}).

\item[---] Let~$\cA$ and~$\cB$ be locally finite varieties of algebras such that every finite \jzs\ has, up to isomorphism, only finitely many liftings, with respect to the functor~$\Conc$, in~$\cB$, and every such lifting is finite. (\emph{The latter condition holds, in particular, if~$\cB$ omits the Tame Congruence Theory types~$\mathbf{1}$ and~$\mathbf{5}$; thus, in particular, if~$\cB$ satisfies a nontrivial congruence lattice identity, see Hobby and McKenzie~\cite{HoMK88}}.) If $\Conc\cA$ is not contained in~$\Conc\cB$, then $\crit{\cA}{\cB}\leq\aleph_2$ (cf. Gillibert~\cite{Gill4}).

\end{itemize}

There is still one completely undefined term in the statement of CLL above, namely ``nice'', as an attribute of a poset. Not every poset qualifies as ``nice''. The definition of ``niceness'', formally \emph{existence of a $\lambda$-lifter} ($\lambda$ is a certain infinite cardinal depending of the data), is relatively complex and thus we shall relegate it to Section~\ref{S:App} (cf. Definition~\ref{D:Lifter}). Nonetheless, it is still possible to describe with relatively little effort which posets are dealt with by CLL, and which are not. Observe that these definitions have, at first sight, little to do with category theory.

\begin{notation}\label{Not:sor}
Let~$X$ be a subset in a poset~$P$. We denote by $\Sor X$ the set of all minimal elements of~$P\Upw X$.
\end{notation}

\begin{definition}\label{D:PJS}
We say that a subset~$X$ in a poset~$P$ is \emph{$\sor$-closed} if $\Sor Y\subseteq X$ for any finite $Y\subseteq X$. The \emph{$\sor$-closure} of a subset~$X$ of a poset~$P$ is the least~$\sor$-closed subset of~$P$ containing~$X$. We say that~$P$ is
\begin{itemize}
\item a \emph{\pjs} if for each finite $X\subseteq P$ there exists a finite $Y\subseteq P$ such that $P\Upw X=P\upw Y$ (cf. Section~\ref{Su:Posets});

\item \emph{supported} if~$P$ is a \pjs\ and the $\sor$-closure of every finite subset of~$P$ is finite;

\item an \emph{\ajs} if~$P$ is a \pjs\ and~$P\dnw a$ is a \js\ for each $a\in P$.
\end{itemize}
\end{definition}

\emph{Every \jzs\ \pup{thus, in particular, every lattice with zero} is an \ajs}. Furthermore, every \ajs\ is supported and every supported poset is a \pjs. Both converse containments fail, as illustrated in Figure~\ref{Fig:ThreePosets}; the leftmost poset is not a \pjs.

\begin{figure}[htb]
\includegraphics{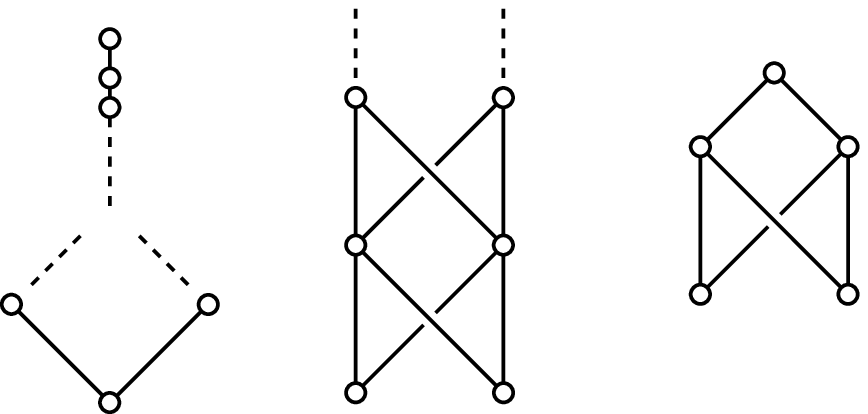}
\caption{Posets that are not \ajs s}
\label{Fig:ThreePosets}
\end{figure}

As mentioned above, ``nice'' in the statement of CLL means having a $\lambda$-lifter, where, roughly speaking, $\lambda$ is an upper bound on the sizes of the categorical data. This brings a severe restriction on the shape of~$P$, involving a connection between the objects of Definition~\ref{D:PJS} and the definition of a lifter (Definition~\ref{D:Lifter}): namely, \emph{If~$P$ has a $\lambda$-lifter, then it is an \ajs}. More precisely, the following result is proved in Gillibert and Wehrung~\cite{Larder}. It relates existence of lifters, \ajs s, and the more familiar-looking (at least for those familiar with large free sets problems) infinite combinatorial statement in~(ii).

\begin{theorem}\label{T:ShapeLiftable}
Let $\lambda$ be an infinite cardinal and let $(X,\bX)$ be a $\lambda$-lifter of a poset~$P$. Put $\kappa:=\card\bX$. Then the following statements hold:
\begin{enumerate}
\item $P$ is a disjoint union of finitely many \ajs s with zero.

\item For every \emph{isotone} map $F\colon[\kappa]^{<\cf(\lambda)}\to[\kappa]^{<\lambda}$, there exists a \emph{one-to-one} map $\sigma\colon P\mono\kappa$ such that
 \[
 (\forall a<b\text{ in }P)\bigl(F\sigma(P\dnw a)\cap\sigma(P\dnw b)
 \subseteq\sigma(P\dnw a)\bigr)\,.
 \]
\end{enumerate}
\end{theorem}

In particular, as none of the posets represented on Figure~\ref{Fig:ThreePosets} is an \ajs, none of them has a lifter. As to the statement of CLL, these posets are not ``nice''.

Item~(ii) of Theorem~\ref{T:ShapeLiftable} bears an intriguing relationship with the problem of \emph{well-foundedness} of liftable posets. The following result is proved in Gillibert and Wehrung~\cite{Larder}.

\begin{theorem}\label{T:WFstrCard}
Let $\lambda$ be an infinite cardinal such that
 \[
 (\forall\mu<\cf(\lambda))\bigl(\mu^{\aleph_0}<\lambda\bigr)\,.
 \]
If a poset has a $\lambda$-lifter, then it is well-founded.
\end{theorem}

In particular, if a poset has a $(2^{\aleph_0})^+$-lifter, then it is well-founded. Hence, the poset $(\omega+1)^{\mathrm{op}}:=\set{0,1,2,\dots}\cup\set{\omega}$, with the total ordering defined by $0>1>2>\cdots>\omega$, has no $(2^{\aleph_0})^+$-lifter. It is also proved in Gillibert and Wehrung~\cite{Larder} that if an ill-founded (i.e., non well-founded) poset has a $\lambda$-lifter, then so does $(\omega+1)^{\mathrm{op}}$. By using Theorem~\ref{T:ShapeLiftable}, we obtain the following.

\begin{corollary}\label{C:w+1lif}
If there exists an ill-founded poset with an $\aleph_1$-lifter, then there exists an infinite cardinal~$\kappa$ such that for each \emph{isotone} $F\colon[\kappa]^{\les\aleph_0}\to[\kappa]^{\les\aleph_0}$, there exists a sequence $\famm{\kappa_n}{n<\omega}$ from~$\kappa$ such that $\kappa_n\notin F(\setm{\kappa_i}{i>n})$ for each $n<\omega$.
\end{corollary}

We do not know whether the final statement of Corollary~\ref{C:w+1lif} can occur at all (cf. Problem~\ref{Pb:SuperErd}). Actually, we do not know whether $(\omega+1)^{\mathrm{op}}$ can have a $\lambda$-lifter, for any infinite cardinal~$\lambda$. By Theorem~\ref{T:WFstrCard}, $\lambda:=\aleph_1$ is the first possible candidate.

\section{Critical points between varieties}\label{S:CritPt}

A good illustration of the various combinatorial principles involved  in Section~\ref{S:CLL} is provided by the theory of \emph{critical points} (cf. Definition~\ref{D:CritPt}). An upper bound for many critical points is provided by the following deep and difficult result, proved in Gillibert~\cite{GillTh,Gill1}.

\begin{theorem}\label{T:CritPtDich}
Let~$\cA$ and~$\cB$ be varieties of \emph{algebras}, with~$\cA$ locally finite and~$\cB$ finitely generated congruence-distributive. If $\Conc\cA\not\subseteq\Conc\cB$, then $\crit{\cA}{\cB}<\nobreak\aleph_\omega$.
\end{theorem}

This result is extended to quasivarieties (and relative congruence lattices) in Gillibert and Wehrung~\cite{Larder}. Moreover, by using the results of Tame Congruence Theory TCT (cf. Hobby and McKenzie~\cite{HoMK88}), it is shown in Gillibert and Wehrung~\cite{Larder} that Theorem~\ref{T:CritPtDich} can be extended to the case where~$\cB$ has finite signature, is finitely generated, and omits both TCT types~$\mathbf{1}$ and~$\mathbf{5}$. In particular, the latter condition holds in case~$\cB$ satisfies a nontrivial congruence lattice identity. The strongest refinement of Theorem~\ref{T:CritPtDich} to date, obtained in Gillibert~\cite{Gill4}, is for~$\cA$ and~$\cB$ both locally finite with $\cB$ ``strongly congruence-proper'' (which holds in case~$\cB$ omits both TCT types~$\mathbf{1}$ and~$\mathbf{5}$); furthermore, the conclusion is strengthened there to $\crit{\cA}{\cB}\leq\aleph_2$. The bound~$\aleph_2$ is sharp.

For lattice varieties, Theorem~\ref{T:CritPtDich} is refined further in Gillibert~\cite{Gill3}. The proof of Theorem~\ref{T:CritLatVar} is extremely involved, and it uses a precursor of CLL (cf. Section~\ref{S:CLL}) called the \emph{Armature Lemma}.

\begin{theorem}[Gillibert]\label{T:CritLatVar}
Let~$\cA$ and~$\cB$ be varieties of \emph{lattices} such that every simple member of~$\cB$ has a prime interval. If~$\cA$ is contained neither in~$\cB$ nor in the dual variety of~$\cB$, then $\Conc\cA\not\subseteq\Conc\cB$ and $\crit{\cA}{\cB}\leq\aleph_2$. The bound~$\aleph_2$ is sharp.
\end{theorem}

It is not known whether the assumption about simple members of~$\cB$ can be removed from the assumptions of Theorem~\ref{T:CritLatVar}.

In Theorem~\ref{T:CritLatVar}, the optimality statement on the bound~$\aleph_2$ follows from earlier results of Plo\v s\v cica~\cite{Plos00, Plos03,Plos04} and Gillibert~\cite{GillTh,Gill1}: Plo\v s\v cica obtains examples with $\crit{\cA}{\cB}=\aleph_0$ and other examples with $\crit{\cA}{\cB}=\aleph_2$, while Gillibert obtains an example with $\crit{\cA}{\cB}=\aleph_1$. In all those examples, $\cA$ and~$\cB$ can be taken \emph{finitely generated modular} lattice varieties and $\cA\supset\cB$. Estimates for further critical points, between more complicated modular lattice varieties, are given in Gillibert~\cite{Gill2}.

Establishing an estimate of the form $\crit{\cA}{\cB}>\aleph_0$ requires an analogue of Proposition~\ref{P:Con1toCtble}, in particular a direct limit along the chain~$\omega$ of all natural numbers. Likewise, establishing an estimate of the form $\crit{\cA}{\cB}>\aleph_1$ requires an analogue of Proposition~\ref{P:Con2toAl1}, in particular a direct limit along a $2$-ladder of cardinality~$\aleph_1$. On the other hand, proving an estimate of the form $\crit{\cA}{\cB}\leq\aleph_m$ requires different techniques, the combinatorial aspects of which we shall outline in Section~\ref{S:Kurat}. In many proofs of such estimates, one first finds a finite diagram~$\vec\bA$ in~$\cA$, indexed by a finite lattice~$P$, such that $\Conc\vec\bA$ is never isomorphic to~$\Conc\vec\bB$ for any $P$-indexed diagram~$\vec\bB$ in~$\cB$; then, using CLL, this diagram counterexample is turned to an object counterexample in~$\cA$.

Before going back, in Section~\ref{S:Kurat}, to the combinatorial aspects of such arguments, we shall see in Section~\ref{S:BowTie} that certain problems may have diagram counterexamples but no known object counterexamples.

\section{A diagram counterexample without any known associated object counterexample}\label{S:BowTie}

The Congruence Lattice Problem (CLP), raised by R.\,P. Dilworth in the forties, asked whether every distributive algebraic lattice is isomorphic to the congruence lattice of a lattice; equivalently, whether every distributive \jzs\ is isomorphic to~$\Conc\bL$ for a lattice~$\bL$. This problem got finally settled in the negative in Wehrung~\cite{CLP}; the sharp bound for the size of the counterexample (viz. $\aleph_2$) was obtained in R\r{u}\v{z}i\v{c}ka~\cite{Ruzi08}. For an account of the problem, see Gr\"atzer~\cite{Grat07}.

Before CLP got settled, Pudl\'ak~\cite{Pudl} asked whether CLP could have a positive, \emph{functorial} solution, that is, a functor~$\Gamma$, from distributive \jzs s and \jzh s to lattices and lattice homomorphisms, such that $\Conc\Gamma(\bS)\cong\bS$ naturally in~$\bS$ for any distributive \jzs\ $\bS$. This problem got solved in the negative (before CLP got settled) in T\r{u}ma and Wehrung~\cite{Bowtie}. While the first version of that negative answer was obtained for congruence lattices of lattices, it got soon extended to a much wider class of structures by using deep results in commutator theory by Kearnes and Szendrei~\cite{KeSz98}. The result can be stated as follows.

\begin{theorem}\label{T:Bowtie}
There exists a diagram~$\cD_{\bowtie}$, indexed by a finite poset, of finite Boolean semilattices and \jzue s, which cannot be lifted, with respect to the~$\Conc$ functor, in any variety satisfying a nontrivial congruence lattice identity.
\end{theorem}

In particular, the diagram~$\cD_{\bowtie}$ cannot be lifted, with respect to the~$\Conc$ functor, by \emph{lattices}, \emph{majority algebras}, \emph{groups}, \emph{loops}, \emph{modules}, and so on. On the other hand, Lampe~\cite{Lamp82} proved that every \jzus\ (distributive or not) is isomorphic to~$\Conc\bG$ for some \emph{groupoid}~$\bG$ (a \emph{groupoid} in universal algebra is just a nonempty set with a binary operation). By Gillibert~\cite{GillTh,Gill1}, this result can be extended to any diagram of \jzus s and \jzuh s indexed by a finite poset, and in particular to~$\cD_{\bowtie}$: that is, \emph{$\cD_{\bowtie}$ can be lifted by a diagram of groupoids}.

We shall now describe this diagram. Consider the \jzuh s $\xe\colon\Pow(1)\to\Pow(2)$,
$\xf_i\colon\Pow(2)\to\Pow(3)$, and $\xu_i\colon\Pow(3)\to\Pow(4)$ (for $i<3$), determined by their values on the atoms of their respective domains:
 \begin{gather*}
 \xe\colon\set{0}\mapsto\set{0,1};\\
 \xf_0\colon
 \begin{cases}
 \set{0}&\mapsto\set{0,1}\\
 \set{1}&\mapsto\set{0,2}
 \end{cases},\qquad
 \xf_1\colon
 \begin{cases}
 \set{0}&\mapsto\set{0,1}\\
 \set{1}&\mapsto\set{1,2}
 \end{cases},\qquad
 \xf_2\colon
 \begin{cases}
 \set{0}&\mapsto\set{0,2}\\
 \set{1}&\mapsto\set{1,2}
 \end{cases},\\
 \xu_0\colon\begin{cases}
 \set{0}&\mapsto\set{0}\\
 \set{1}&\mapsto\set{1,3}\\
 \set{2}&\mapsto\set{2,3}\\
 \end{cases},\qquad
 \xu_1\colon\begin{cases}
 \set{0}&\mapsto\set{0,3}\\
 \set{1}&\mapsto\set{1}\\
 \set{2}&\mapsto\set{2,3}\\
 \end{cases},\qquad
 \xu_2\colon\begin{cases}
 \set{0}&\mapsto\set{0,3}\\
 \set{1}&\mapsto\set{1,3}\\
 \set{2}&\mapsto\set{2}\\
 \end{cases}.
 \end{gather*}

The diagram~$\cD_{\bowtie}$ is represented on Figure~\ref{Fig:BowTie4}.

\begin{figure}[htb]
 \[
 {
 \def\labelstyle{\displaystyle}
 \xymatrixrowsep{2pc}\xymatrixcolsep{1.5pc}
 \xymatrix{ & \Pow(4) & \\
 \Pow(3)\ar[ru]^{\xu_0} & \Pow(3)\ar[u]|-(.45){\strut\xu_1} &
 \Pow(3)\ar[lu]_{\xu_2}\\
 & &\\
 \Pow(2)\ar[uu]^{\xf_0}\ar[ruu]^(.35){\xf_0}\ar[rruu]|-(.15){\xf_0}&
 \Pow(2)\ar[luu]|-(.7){\strut\xf_1}\ar[uu]|-(.7){\strut\xf_1}
 \ar[ruu]|-(.7){\strut\xf_1}
 &
 \Pow(2)\ar[lluu]|-(.15){\xf_2}\ar[luu]_(.35){\xf_2}\ar[uu]_{\xf_2}\\
 &\Pow(1)\ar[lu]^{\xe}\ar[u]|-(.45){\strut\xe}\ar[ru]_{\xe} &
 }}
 \]
\caption{The diagram $\cD_{\bowtie}$}\label{Fig:BowTie4}
\end{figure}

\begin{figure}[htb]
\includegraphics{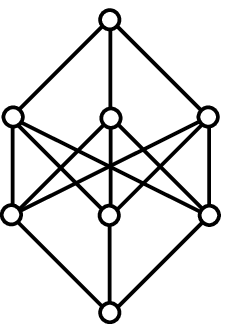}
\caption{The underlying poset of $\cD_{\bowtie}$}
\label{Fig:K33Pos}
\end{figure}

The underlying poset of~$\cD_{\bowtie}$ is represented on Figure~\ref{Fig:K33Pos}. It does not have any lifter (because it is \emph{not an \ajs}). Thus CLL is not sufficient to turn it into an object counterexample. In particular, it is still unknown whether every distributive \jzs\ is isomorphic to $\Conc\bM$ for some \emph{majority algebra}~$\bM$. (A \emph{majority algebra} is a nonempty set~$M$, endowed with a ternary operation~$m$, such that $m(x,x,y)=m(x,y,x)=m(y,x,x)=x$ for all $x,y\in M$. It is well-known that the congruence lattice of a majority algebra is distributive.) Therefore, the representation problem of distributive \jzs s as compact congruence semilattices of majority algebras has a diagram counterexample (namely~$\cD_{\bowtie}$) but no known object counterexample. It is very difficult to prove that the containment of the congruence class of all lattices into the one of all majority algebras is \emph{proper}; this is done in Plo\v{s}\v{c}ica~\cite{Plos08}.

\section{The Kuratowski index of a finite poset}\label{S:Kurat}

The complexity of the definition of a lifter (Definition~\ref{D:Lifter}) makes its verification quite unpractical, even for easily described finite posets. The following much more user-friendly variant is introduced in Gillibert and Wehrung~\cite{GiWe1}.

\begin{notation}\label{Not:leadsto}
For infinite cardinals $\kappa$, $\lambda$ and a poset~$P$, let $(\kappa,{<}\lambda)\leadsto P$ hold if for every mapping $F\colon\Pow(\kappa)\to[\kappa]^{<\lambda}$ there exists a one-to-one map $\sigma\colon P\mono\kappa$ such that
 \[
 (\forall x<y\text{ in }P)
 \bigl(F\sigma(P\dnw x)\cap\sigma(P\dnw y)\subseteq\sigma(P\dnw x)\bigr)\,.
 \]
\end{notation}

As proved in Gillibert and Wehrung~\cite{GiWe1}, if~$P$ is \emph{lower finite}, then it is sufficient to replace $P\dnw z$ by $\J(P)\dnw z$ in the statement above. This makes the verification of the statement $(\kappa,{<}\lambda)\leadsto P$ much more convenient on given finite posets.

The following result from Gillibert and Wehrung~\cite{Larder} tells exactly when a small enough poset has a lifter (cf. Definition~\ref{D:Lifter}), and gives an estimate, in terms of the $(\kappa,{<}\lambda)\leadsto P$ notation, of the size of such a lifter.

\begin{theorem}\label{T:CharLift}
Let~$\lambda$ and~$\kappa$ be infinite cardinals and let~$P$ be a lower finite poset in which every element has less than~$\cf(\lambda)$ upper covers. Then the following are equivalent:
\begin{enumerate}
\item $P$ has a $\lambda $-lifter $(X,\bX)$ such that~$\bX$ consists of all principal ideals of~$X$ and $\card X=\kappa$, while~$X$ is a lower finite \ajs.

\item $P$ has a $\lambda$-lifter $(X,\bX)$ such that $\card\bX=\kappa$.

\item $P$ is a finite disjoint union of \ajs s with zero, and $(\kappa,{<}\lambda)\leadsto\nobreak P$.
\end{enumerate}
\end{theorem}

The following definition and result can be found in Gillibert and Wehrung~\cite{GiWe1}.

\begin{definition}\label{D:KurInd}
The \emph{Kuratowski index} of a finite poset~$P$, denoted by $\kur(P)$, is defined as~$0$ if~$P$ is pairwise incomparable, and, otherwise, as the least $n>0$ if it exists such that $(\lambda^{+(n-1)},{<}\lambda)\leadsto P$ for each infinite cardinal~$\lambda$.
\end{definition}

\begin{proposition}\label{P:kur(P)exists}
The Kuratowski index of a finite poset~$P$ always exists, and $\kur(P)\leq\dim(P)\leq\wdt\J(P)\leq\card\J(P)$. Furthermore, if~$P$ is a \js, then $\brd(P)\leq\kur(P)$.
\end{proposition}

The following corollary can be found in Gillibert and Wehrung~\cite{Larder}.

\begin{corollary}\label{C:Estimkur(P)}
For every infinite cardinal~$\lambda$, every nontrivial finite \ajs\ $P$ with zero has a $\lambda$-lifter $(X,\bX)$ with $\card X=\card\bX=\lambda^{+(\kur(P)-1)}$.
\end{corollary}

There are finite posets for which the Kuratowski index is easy to calculate, see Gillibert and Wehrung~\cite{GiWe1}.

\begin{proposition}\label{P:EasyKurP}
The following statements hold, for any finite posets~$P$ and~$Q$.
\begin{enumerate}
\item $\kur(P)=1$ whenever~$P$ is a nontrivial finite tree.

\item $P\subseteq Q$ implies that $\kur(P)\leq\kur(Q)$.

\item $\kur(P\times Q)\leq\kur(P)+\kur(Q)$.

\item Let $n$ be a natural number. If~$P$ is a product of~$n$ nontrivial trees, then $\kur(P)=n$.
\end{enumerate}
\end{proposition}

In particular, it follows from Proposition~\ref{P:EasyKurP}(iv) that $\kur(\two^n)=n$ for every natural number~$n$; this is part of Kuratowski's Free Set Theorem (cf. Kuratowski~\cite{Kura51}).

However, there are also easily described finite lattices for which the Kuratowski index is unknown. For example, consider the finite lattices~$P$ and~$Q$ represented in Figure~\ref{Fig:lattices}.
\begin{figure}[htp]
\includegraphics{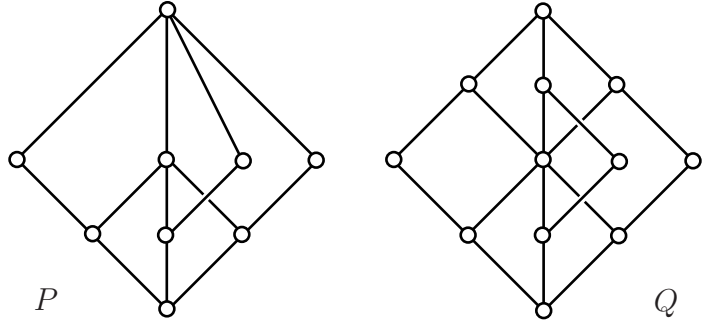}
\caption{Two lattices of breadth two and order-dimension three}\label{Fig:lattices}
\end{figure}
Both~$P$ and~$Q$ have breadth two and order-dimension three. As~$P$ embeds, as a poset, into~$Q$, we obtain the inequalities
 \[
 2\leq\kur(P)\leq\kur(Q)\leq 3\,.
 \]
We do not know which of those inequalities are actually equalities. A simple infinite combinatorial statement equivalent to $\kur(P)=2$ (resp., $\kur(Q)=2$) is given in Gillibert and Wehrung~\cite[Section~6]{GiWe1}.

\section{Large free sets and truncated cubes}\label{S:LargeFree}

Establishing upper bounds for various critical points, due either to Gillibert or to Plo\v{s}\v{c}ica, makes a heavy use of \emph{large free sets}. We recall the following notation, see Erd\H os \emph{et al.}~\cite{EHMR}.

\begin{notation}\label{Not:arrRel}
For cardinals $\kappa$, $\lambda$, $\rho$, and~$\mu$, let $(\kappa,\rho,\lambda)\rightarrow\mu$ hold if for every $F\colon[\kappa]^{\rho}\to[\kappa]^{<\lambda}$, there exists $H\in[\kappa]^{\mu}$ such that $F(X)\cap H\subseteq X$ for each $X\in[H]^\rho$ \pup{we say that~$H$ is \emph{free with respect to~$F$}}.
\end{notation}

For example, Kuratowski's Free Set Theorem (cf. Kuratowski~\cite{Kura51}) states that $(\kappa,n,\lambda)\rightarrow n+1$ holds if{f} $\kappa\geq\lambda^{+n}$, for all infinite cardinals~$\kappa$, $\lambda$ and every natural number~$n$.

For $1\leq r\leq m$, we define the \emph{truncated $m$-dimensional cube}
 \[
 \rB_m({\les}r):=\setm{X\in\Pow(m)}{\text{either }
 \card X\leq r\text{ or }X=m}\,,
 \]
endowed with containment. The symbols $(\kappa,\rho,\lambda)\rightarrow\mu$ and $(\kappa,{<}\lambda)\leadsto P$ are related, in Gillibert and Wehrung~\cite{GiWe1}, as follows.

\begin{proposition}\label{P:rightarr2leadsto}
Let $r$ and $m$ be integers with $1\leq r<m$ and let~$\kappa$ and~$\lambda$ be infinite cardinals. Then $(\kappa,{<}\lambda)\leadsto\rB_m({\les}r)$ if{f} $(\kappa,r,\lambda)\rightarrow m$.
\end{proposition}

Much is already known about the relation $(\kappa,r,\lambda)\rightarrow m$.

\begin{theorem}\label{T:rightarrRem}
Let~$\lambda$ be an infinite cardinal.
\begin{enumerate}
\item If $\mathsf{GCH}$ holds, then $(\lambda^{+r},r,\lambda)\rightarrow\lambda^+$ for every integer $r\geq1$.

\item $(\lambda^{+2},2,\lambda)\rightarrow m$ for every integer $m>0$. Hence $\kur\bigl(\rB_m({\les}2)\bigr)=3$ for all $m\geq3$.

\item $(\lambda^{+3},3,\lambda)\rightarrow m$ for every integer $m>0$. Hence $\kur\bigl(\rB_m({\les}3)\bigr)=4$ for all $m\geq4$.
\end{enumerate}
\end{theorem}

Theorem~\ref{T:rightarrRem}(i) is proved in Erd\H os \emph{et al.}~\cite[Theorem~45.5]{EHMR}.
Theorem~\ref{T:rightarrRem}(ii) is proved in Hajnal and M\'at\'e~\cite{HaMa75}, see also Erd\H os \emph{et al.}~\cite[Theorem~46.2]{EHMR}. The first use of that property for estimating critical points is due to Plo\v s\v cica~\cite{Plos00}. Theorem~\ref{T:rightarrRem}(iii) is attributed, in Erd\H os \emph{et al.}~\cite[Theorem~46.2]{EHMR}, to Hajnal.

The situation for higher truncated cubes is more strange. Set $t_0:=5$, $t_1:=7$, and, for each~$n>0$, let $t_{n+1}\rightarrow(t_n,7)^5$. The latter symbol means that for every $f\colon[t_{n+1}]^5\to\set{0,1}$, either there exists $H\in[t_{n+1}]^{t_n}$ such that $f``[H]^5=\set{0}$ or there exists $H\in[t_{n+1}]^7$ such that $f``[H]^5=\set{1}$. The sequence $\famm{t_n}{n<\omega}$ exists, due to Ramsey's Theorem. The following result is established in Komj\'ath and Shelah~\cite{KoSh00}.

\begin{theorem}\label{T:KoSh00}
For every $n>0$, there exists a generic extension of the universe in which $(\aleph_n,4,\aleph_0)\not\rightarrow t_n$.
\end{theorem}

In particular, there exists a generic extension of the universe with $(\aleph_4,4,\aleph_0)\not\rightarrow\nobreak t_4$. Hence $\kur\bigl(\rB_{t_4}({\les}4)\bigr)=5$ in any universe with~$\mathsf{GCH}$, while $\kur\bigl(\rB_{t_4}({\les}4)\bigr)\geq6$ in some generic extension. In particular, \emph{The assignment $P\mapsto\kur(P)$, for a finite lattice~$P$, is not absolute} (in the set-theoretical sense).

By using Propositions~\ref{P:kur(P)exists}, \ref{P:EasyKurP}, and~\ref{P:rightarr2leadsto}, it is possible to get new estimates for large free sets results. The first such estimate is the following. For any positive integer~$n$, it is not hard to verify that $\brd\bigl(\rB_{n+2}({\les}n)\bigr)=\dim\bigl(\rB_{n+2}({\les}n)\bigr)=n+1$. By Proposition~\ref{P:kur(P)exists}, it follows that $\kur\bigl(\rB_{n+2}({\les}n)\bigr)=n+1$. We obtain the following result of Gillibert~\cite{GillTh}, see also Gillibert and Wehrung~\cite{GiWe1} for a self-contained proof.

\begin{theorem}\label{T:Gilln+2}
The relation $(\lambda^{+n},n,\lambda)\rightarrow n+2$ holds for each infinite cardinal~$\lambda$ and each positive integer~$n$.
\end{theorem}

In particular, $(\aleph_4,4,\aleph_0)\rightarrow 6$, which answers a question raised in Erd\H os \emph{et al.}~\cite[page~285]{EHMR}. The previously known bound, namely $(\aleph_4,4,\aleph_0)\rightarrow 5$, follows from Kuratowski's Free Set Theorem. The next open question is, naturally, whether $(\aleph_4,4,\aleph_0)\rightarrow 7$. This cannot go on forever, as we know, from Komj\'ath and Shelah's result, that $(\aleph_4,4,\aleph_0)\not\rightarrow t_4$ in some generic extension.

The large free set relation $(\aleph_4,4,\aleph_0)\rightarrow 6$ is obtained by finding an upper bound (in fact, here, an exact evaluation) of the order-dimension of $\rB_6({\les}4)$ (namely, $5$). Now the problem of estimating the order-dimensions of truncated cubes~$\rB_m({\les}r)$ is a whole topic, started by Dushnik~\cite{Dush50} and going on in various works such as Brightwell \emph{et al.}~\cite{BKKT94}, F\"uredi and Kahn~\cite{FuKa86}, Kierstead~\cite{Kier96,Kier99}, Spencer~\cite{Spen71}. Using the various estimates available, we obtain in Gillibert and Wehrung~\cite{GiWe1} new large free sets relations, of which a sample is the following.

\goodbreak

\begin{proposition}\label{P:FreeSetSample}\hfill
\begin{enumerate}
\item $(\aleph_7,4,\aleph_0)\rightarrow10$.

\item $(\aleph_9,5,\aleph_0)\rightarrow12$.

\item $(\aleph_{109},4,\aleph_0)\rightarrow257$.

\item $(\aleph_{210},4,\aleph_0)\rightarrow32{,}768$.
\end{enumerate}
\end{proposition}

(\emph{Although it looks quite strange that the size of the free set is one more than a power of~$2$ in~\textup{(iii)} while it is a power of two in~\textup{(iv)}, this seems to be due to the formulas involved having nothing to do with each other}.)
Proposition~\ref{P:FreeSetSample}(i,ii) uses Dushnik's exact estimate from~\cite{Dush50}, while~(iii) uses F\"uredi and Kahn's estimate from~\cite{FuKa86}, and~(iv) uses the estimate attributed in Spencer~\cite{Spen71} to Hajnal. Among those estimates, the best asymptotic behavior is enjoyed by Hajnal and Spencer's one. This estimate makes it possible to prove that $(\aleph_n,r,\aleph_0)\rightarrow E(n,r)$ with
 \[
 \lg\lg E(n,r)\sim\frac{n}{r2^r\log 2}\text{ as }n\gg r\gg 0
 \]
(where $\lg x$ denotes the base~$2$ logarithm while $\log x$ denotes the natural logarithm).

\section{Appendix: norm-coverings and lifters}\label{S:App}

In the present section we shall give formal definitions for some of the objects involved in Section~\ref{S:CLL}. We restrict ourselves to infinite combinatorial objects (\emph{norm-coverings}, \emph{lifters}) and omit the more complicated categorical objects (\emph{larders}), for which formal definitions can be found in Gillibert and Wehrung~\cite{Larder}. We remind the reader that \pjs s and supported posets are defined in Section~\ref{S:CLL}.

We start with an auxiliary definition.

\begin{definition}\label{D:NormCov}
Let~$P$ be a poset.
\begin{itemize}
\item A \emph{norm-covering} of~$P$ consists of a \pjs~$X$ together with an isotone map $\partial\colon X\to P$.

\item An ideal~$\bx$ of~$X$ is \emph{sharp} if $\setm{\partial x}{x\in\bx}$ has a largest element, then denoted by~$\partial\bx$.

\item Then we set $\bX^=:=\setm{\bx\in\bX}{\partial\bx\text{ is not maximal}}$, for every set~$\bX$ of sharp ideals of~$X$.
\end{itemize}
\end{definition}

Now we can define lifters.

\begin{definition}\label{D:Lifter}
For an infinite cardinal~$\lambda$, a \emph{$\lambda$-lifter} of a poset~$P$ is a pair $(X,\bX)$, where~$X$ is a norm-covering of~$P$, $\bX$ is a set of sharp ideals of~$X$, and
\begin{enumerate}
\item Every principal ideal of~$\bX^=$ has cardinality $<\cf(\lambda)$.

\item For every map $S\colon\bX^=\to[X]^{<\lambda}$ there exists an isotone section~$\sigma\colon P\into\bX$ of~$\partial$ such that $S\sigma(a)\cap\sigma(b)\subseteq\sigma(a)$ for all $a<b$ in~$P$.

\item If $\lambda=\aleph_0$, then $X$ is supported.
\end{enumerate}
\end{definition}

\section{Open problems}\label{S:Pbs}

\subsection{Ladders}
Let us first restate Ditor's Problem~\cite{Dito84} about ladders (cf. Section~\ref{S:Con3}).

\begin{problem}\label{Pb:DitorLadd}
Does there exist a $3$-ladder of cardinality~$\aleph_2$?
\end{problem}

So far, we know that a positive solution to Problem~\ref{Pb:DitorLadd} is \emph{consistent with~$\mathsf{ZFC}$}, cf. Theorem~\ref{T:Ex3Ladd}. Even more, every universe of~$\mathsf{ZFC}$ has a ccc generic extension and a countably closed generic extension both satisfying the existence of a $3$-ladder of cardinality~$\aleph_2$.

In fact, Ditor asks in~\cite{Dito84} whether there exists a $k$-ladder of cardinality~$\aleph_{k-1}$ (which is the upper bound given by Kuratowski's Free Set Theorem), for $k\geq3$. Because of Theorem~\ref{T:Ex3Ladd}, we can also ask the following weaker variant.

\begin{problem}\label{Pb:kladd}
Let $k\geq3$ be an integer. Is the existence of a $k$-ladder of cardinality~$\aleph_{k-1}$ consistent with~$\mathsf{ZFC}$?
\end{problem}

\subsection{Well-foundedness of liftable posets}

We ask in Gillibert and Wehrung~\cite{Larder} whether, for any infinite cardinal~$\lambda$, any poset with a $\lambda$-lifter is well-founded. Equivalently, we ask whether $(\omega+1)^{\mathrm{op}}$ has no $\lambda$-lifter, for any infinite cardinal~$\lambda$. As explained at the end of Section~\ref{S:CLL}, this is related to the following question.

\begin{problem}\label{Pb:SuperErd}
Does there exist an infinite cardinal~$\kappa$ such that for any \emph{isotone} mapping $F\colon[\kappa]^{\les\aleph_0}\to[\kappa]^{\les\aleph_0}$, there exists a sequence $\famm{\kappa_n}{n<\omega}$ from~$\kappa$ such that $\kappa_n\notin F(\setm{\kappa_i}{i>n})$ for each $n<\omega$?
\end{problem}

Recall that the isotonicity assumption simply means that $x\subseteq y$ implies that $F(x)\subseteq F(y)$ for any countable subsets~$x$ and~$y$ of~$\kappa$. It can be easily seen that if we remove this isotonicity assumption, then the corresponding problem has no solution (cf. the end of the proof of Proposition~3-6.3 in Gillibert and Wehrung~\cite{Larder}). Fred Galvin and Pierre Gillibert observed independently that every cardinal~$\kappa$ solving positively Problem~\ref{Pb:SuperErd} should be greater than or equal to~$\aleph_\omega$.

\subsection{Critical points}

Our first problem about critical points asks whether these are \emph{absolute} (in the set-theoretical sense).

\begin{problem}\label{Pb:AbsCrPt}
Are there finitely generated varieties~$\cA$ and~$\cB$ of algebras, distinct natural numbers~$m$ and~$n$, and models~$\mathbf{M}_1$ and~$\mathbf{M}_2$ of~$\mathsf{ZFC}$ set theory, such that $\crit{\cA}{\cB}=\aleph_m$ in~$\mathbf{M}_1$ while~$\crit{\cA}{\cB}=\aleph_n$ in~$\mathbf{M}_2$?
\end{problem}

A related question is raised in Gillibert and Wehrung~\cite{Larder}, namely whether critical points can be calculated by a recursive algorithm. This would of course imply a negative answer to Problem~\ref{Pb:AbsCrPt}.

\begin{problem}\label{Pb:CrPtAl3}
Are there \pup{finitely generated} varieties~$\cA$ and~$\cB$ \pup{on finite signatures} such that $\crit{\cA}{\cB}=\aleph_3$?
\end{problem}

It is plausible that the construction of an example with $\crit{\cA}{\cB}=\aleph_3$ would involve $3$-ladders (cf. Problem~\ref{Pb:DitorLadd}) together with some form of three-dimensional amalgamation (see Section~\ref{S:Con3}).

\subsection{Kuratowski index}

We remind the reader that a challenging question, raised in Gillibert and Wehrung~\cite[Section~6]{GiWe1}, is to determine the Kuratowski indexes of the two finite lattices~$P$ and~$Q$ represented in Figure~\ref{Fig:lattices}. Another puzzling question is the following (cf. Notation~\ref{Not:arrRel}).

\begin{problem}\label{Pb:rightarrow7}
Does $(\aleph_4,4,\aleph_0)\rightarrow 7$?
\end{problem}

An even far more ambitious question is to ask whether one can determine the largest integer~$m$ such that~$\mathsf{ZFC}$ proves the relation $(\aleph_4,4,\aleph_0)\rightarrow m$. All we know so far is that $6\leq m<t_4$ (cf. Section~\ref{S:LargeFree}). This leaves plenty of room.


\begin{thebibliography}{99}
\bibitem{BKKT94}
G.\,R. Brightwell, H.\,A. Kierstead, A.\,V. Kostochka, and W.\,T. Trotter,
\emph{The dimension of suborders of the Boolean lattice}, Order~\textbf{11} (1994), 127--134.

\bibitem{BTFr09}
A.\,D. Brooke-Taylor and S.-D. Friedman,
\emph{Large cardinals and gap-1 morasses}, 
Ann. Pure Appl. Logic~\textbf{159}, no.~1-2 (2009), 71--99.

\bibitem{BuMD78}
S. Bulman-Fleming and K. McDowell,
\emph{Flat semilattices}, Proc. Amer. Math. Soc.~\textbf{72} (1978), no.~2, 228--232.

\bibitem{Devl84}
K. Devlin, ``Constructibility'',
Perspectives in Mathematical Logic. Springer-Verlag, Berlin, 1984. xi+425~p. ISBN: 3-540-13258-9.

\bibitem{Dito84}
S.\,Z. Ditor,
\emph{Cardinality questions concerning semilattices of finite breadth},
Discrete Math. \textbf{48} (1984), 47--59.

\bibitem{Dobb86}
H. Dobbertin,
\emph{Vaught measures and their applications in lattice theory},
J. Pure Appl. Algebra \textbf{43}, no.~1 (1986), 27--51.

\bibitem{Dush50}
B. Dushnik,
\emph{Concerning a certain set of arrangements}, 
Proc. Amer. Math. Soc.~\textbf{1}, (1950). 788--796.

\bibitem{EHMR}
P. Erd\H{o}s, A. Hajnal, A. M\'at\'e, and R. Rado,
``Combinatorial Set Theory: Partition Relations for Cardinals''.
Studies in Logic and the Foundations of Mathematics~\textbf{106}. North-Holland Publishing Co., Amsterdam, 1984. 347~p. ISBN: 0-444-86157-2

\bibitem{FuKa86}
Z. F\"uredi and J. Kahn,
\emph{On the dimensions of ordered sets of bounded degree},
Order~\textbf{3}, no.~1 (1986), 15--20.

\bibitem{GillTh}
P. Gillibert, ``Points critiques de couples de vari\'et\'es d'alg\`ebres'', Doc\-to\-rat de l'U\-ni\-ver\-si\-t\'e de Caen, December~8, 2008. Available online at \texttt{http://tel.archives-ouvertes.fr/tel-00345793}\,.

\bibitem{Gill1}
P. Gillibert, \emph{Critical points of pairs of varieties of algebras}, Internat. J. Algebra Comput.~\textbf{19}, no.~1 (2009), 1--40.

\bibitem{Gill2}
P. Gillibert,
\emph{Critical points between varieties generated by subspace lattices of vector spaces}, J. Pure Appl. Algebra~\textbf{214} (2010), 1306--1318.

\bibitem{GilCat}
P Gillibert,
\emph{Categories of partial algebras for critical points between varieties of algebras}, preprint 2010. Available online at \texttt{http://hal.archives-ouvertes.fr/hal-00544601}\,.

\bibitem{Gill3}
P. Gillibert,
\emph{The possible values of critical points between varieties of lattices}, preprint 2010. Available online at \texttt{http://hal.archives-ouvertes.fr/hal-00468048}\,.

\bibitem{Gill4}
P. Gillibert,
\emph{The possible values of critical points between strongly congruence-proper varieties of algebras}, manuscript 2011.

\bibitem{GiWe1}
P. Gillibert and F. Wehrung, \emph{An infinite combinatorial statement with a poset parameter}, Combinatorica, to appear. Available online at\newline \texttt{http://hal.archives-ouvertes.fr/hal-00364329}\,.

\bibitem{Larder}
P. Gillibert and F. Wehrung, \emph{{}From objects to diagrams for ranges of functors}, preprint 2010. Available online at \texttt{http://hal.archives-ouvertes.fr/hal-00462941}\,.

\bibitem{GoWe01}
K.\,R. Goodearl and F. Wehrung,
\emph{Representations of distributive semilattices in ideal lattices of various algebraic structures}, Algebra Universalis~\textbf{45} (2001), no.~1, 71--102.

\bibitem{GLT2}
G. Gr\"atzer,
``General Lattice Theory. Second edition'', new appendices by the
author with B.\,A. Davey, R. Freese, B. Ganter,
M. Greferath, P. Jipsen, H.\,A. Priestley, H. Rose, E.\,T. Schmidt,
S.\,E. Schmidt, F. Wehrung, and R. Wille. Birkh\"auser Verlag, Basel,
1998. xx+663~p. ISBN: 3-7643-5239-6 (Basel), 0-8176-5239-6 (Boston).

\bibitem{Grat07}
G. Gr\"atzer,
\emph{Two problems that shaped a century of lattice theory}, 
Notices Amer. Math. Soc.~\textbf{54} (2007), no.~6, 696--707.

\bibitem{GLW}
G. Gr\"atzer, H. Lakser, and F. Wehrung,
\emph{Congruence amalgamation of lattices}, Acta Sci. Math. (Szeged)~\textbf{66} (2000), 339--358.

\bibitem{HaMa75}
A. Hajnal and A. M\'at\'e,
\emph{Set mappings, partitions, and chromatic numbers}. Logic Colloquium '73 (Bristol, 1973), p. 347--379. Studies in Logic and the Foundations of Mathematics, Vol.~\textbf{80}, North-Holland, Amsterdam, 1975.

\bibitem{HoMK88}
D. Hobby and R.\,N. McKenzie,
``The Structure of Finite Algebras''. Contemporary Mathematics, \textbf{76}. American Mathematical Society, Providence, RI, 1988. xii+203~p. ISBN: 0-8218-5073-3 .

\bibitem{Huhn89a}
A.\,P. Huhn,
\emph{On the representation of algebraic distributive lattices II},
Acta Sci.\ Math.\ (Szeged) \textbf{53} (1989), 3--10.

\bibitem{Huhn89b}
A.\,P. Huhn,
\emph{On the representation of algebraic distributive lattices III},
Acta Sci.\ Math.\ (Szeged) \textbf{53} (1989), 11--18.

\bibitem{Jech03}
T. Jech,
``Set Theory.'' The third millennium edition, revised and expanded. Springer Monographs in Mathematics. Springer-Verlag, Berlin, 2003. xiv+769~p. ISBN: 3-540-44085-2

\bibitem{KeSz98}
K.\,A. Kearnes and \'A. Szendrei,
\emph{The relationship between two commutators},
Internat. J. Algebra Comput. \textbf{8}, no.~4 (1998), 497--531.

\bibitem{Kier96}
H.\,A. Kierstead,
\emph{On the order-dimension of $1$-sets versus $k$-sets},
J. Combin. Theory Ser. A~\textbf{73} (1996), 219--228.

\bibitem{Kier99}
H.\,A. Kierstead,
\emph{The dimension of two levels of the Boolean lattice},
Discrete Math.~\textbf{201} (1999), 141--155.

\bibitem{KoSh00}
P. Komj\'ath and S. Shelah,
\emph{Two consistency results on set mappings}, J. Symbolic Logic~\textbf{65} (2000), 333--338.

\bibitem{KuVa}
``Handbook of Set-Theoretic Topology'', Edited by Kenneth Kunen and Jerry E. Vaughan. North-Holland Publishing Co., Amsterdam, 1984. vii+1273~p. ISBN: 0-444-86580-2.

\bibitem{Kura51}
C. Kuratowski,
\emph{Sur une caract\'erisation des alephs},
Fund. Math. \textbf{38} (1951), 14--17.

\bibitem{Lamp82}
W.\,A. Lampe,
\emph{Congruence lattices of algebras of fixed similarity type. II}, Pacific J. Math. \textbf{103} (1982), 475--508.

\bibitem{Plos00}
M. Plo\v s\v cica,
\emph{Separation properties in congruence lattices of lattices},
Colloq. Math. \textbf{83} (2000), 71--84.

\bibitem{Plos03}
M. Plo\v s\v cica,
\emph{Dual spaces of some congruence lattices},
Topology and its Applications \textbf{131} (2003), 1--14.

\bibitem{Plos04}
M. Plo\v s\v cica,
\emph{Separation in distributive congruence lattices},
Algebra Universalis \textbf{49}, no.~1 (2004), 1--12.

\bibitem{Plos08}
M. Plo\v s\v cica,
\emph{Non-representable distributive semilattices},
J. Pure Appl. Algebra~\textbf{212}, no.~11 (2008), 2503--2512.

\bibitem{Pudl}
P. Pudl\'ak,
\emph{On congruence lattices of lattices},
Algebra Universalis \textbf{20} (1985), 96--114.

\bibitem{Ruzi08}
P.  R\r{u}\v{z}i\v{c}ka,
\emph{Free trees and the optimal bound in Wehrung's theorem}, Fund. Math.~\textbf{198}, no.~3 (2008), 217--228.

\bibitem{RTW}
P. R\r{u}\v{z}i\v{c}ka, J. T\r{u}ma, and F. Wehrung,
\emph{Distributive congruence lattices of congruence-permutable algebras},
J. Algebra \textbf{311}, no.~1 (2007), 96--116.

\bibitem{Spen71}
J. Spencer, \emph{Minimal scrambling sets of simple orders}, Acta Math. Acad. Sci. Hungar.~\textbf{22} (1971/72), 349--353.

\bibitem{CLPSurv}
J. T\r{u}ma and F. Wehrung,
\emph{A survey of recent results on congruence lattices of lattices},
Algebra Universalis \textbf{48}, no.~4 (2002), 439--471.

\bibitem{Bowtie}
J. T\r{u}ma and F. Wehrung,
\emph{Congruence lifting of diagrams of finite Boolean semilattices
requires large congruence varieties}, Internat. J. Algebra Comput. \textbf{16}, no. 3 (2006), 541--550.

\bibitem{Vell82}
D.\,J. Velleman,
\emph{Morasses, diamond, and forcing},
Ann. Math. Log.~\textbf{23} (1982), 199--281.

\bibitem{Vell84a}
D.\,J. Velleman,
\emph{Simplified morasses}, J. Symbolic Logic~\textbf{49}, no.~1 (1984), 257--271. 

\bibitem{WReg}
F. Wehrung,
\emph{Representation of algebraic distributive lattices with
$\aleph_1$ compact elements as ideal lattices of regular rings},
Publ. Mat. (Barcelona) \textbf{44} (2000), 419--435.

\bibitem{CLP}
F. Wehrung,
\emph{A solution to Dilworth's Congruence Lattice Problem},
Adv. Math.~\textbf{216}, no.~2 (2007), 610--625.

\bibitem{3ladd}
F. Wehrung,
\emph{Large semilattices of breadth three}, Fund. Math.~\textbf{208}, no.~1 (2010), 1--21.

\bibitem{Banasch2}
F. Wehrung,
\emph{A non-coordinatizable sectionally complemented modular lattice with a large J\'onsson four-frame}, Adv. in Appl. Math., to appear. Available online at \texttt{http://hal.archives-ouvertes.fr/hal-00462951}\,.

\bibitem{VLift}
F. Wehrung,
\emph{Lifting defects for nonstable K${}_0$-theory of exchange rings and C*-algebras}, preprint 2011. Available online at \texttt{http://hal.archives-ouvertes.fr/hal-00559268}\,.

\bibitem{Weiss}
W. Weiss,
``Versions of Martin's Axiom'',
Chapter~19 in \cite{KuVa}, 827--886.

\end{thebibliography}
\end{document}